\documentclass{bmcart}

\usepackage[utf8]{inputenc} 
\usepackage{amsmath,amsthm,latexsym,amssymb,mathrsfs}

\usepackage{graphicx}

\usepackage{multicol,amsmath,amssymb,amstext,graphicx}
\usepackage{setspace}
\usepackage{esint}
\usepackage{subfig}
\usepackage{multirow}
\usepackage{lipsum}

\startlocaldefs
\usepackage{enumitem}		
\usepackage{braille}

\@ifundefined{lettrine}{\usepackage{lettrine}}{}
 \let\footnote=\endnote
\theoremstyle{plain}
\ifx\thechapter\undefined
\newtheorem{thm}{\protect\theoremname}
\else
\newtheorem{thm}{\protect\theoremname}[chapter]

\fi
  \theoremstyle{plain}
  \newtheorem{lem}[thm]{\protect\lemmaname}
  \theoremstyle{remark}
  \newtheorem{rem}[thm]{\protect\remarkname}

\@ifundefined{showcaptionsetup}{}{%
 \PassOptionsToPackage{caption=false}{subfig}}
\usepackage{subfig}

\makeatother

  \providecommand{\lemmaname}{Lemma}
  \providecommand{\remarkname}{Remark}
\providecommand{\theoremname}{Theorem}

\endlocaldefs

\begin{document}

\begin{frontmatter}

\begin{fmbox}
\dochead{Research}


\title{On the regularization of solution of an inverse ultraparabolic equation
associated with perturbed final data}


\author[
   addressref={aff1},                   
   corref={aff1},                       
   email={vakhoa.hcmus@gmail.com}   
]{\inits{VAK}\fnm{Vo Anh} \snm{Khoa}}
\author[
   addressref={aff1},
   email={tronglanmath@gmail.com}
]{\inits{LTL}\fnm{Le Trong} \snm{Lan}}
\author[
   addressref={aff1,aff2},
   email={thnguyen2683@gmail.com}
]{\inits{NTH}\fnm{Nguyen Huy} \snm{Tuan}}
\author[
   addressref={aff1},
   email={tthung1992@gmail.com}
]{\inits{TTH}\fnm{Tran The} \snm{Hung}}


\address[id=aff1]{
  \orgname{Department of Mathematics and Computer Science, Ho Chi Minh City University of Science}, 
  \street{227 Nguyen Van Cu Street, District 5},                     %
  \city{Ho Chi Minh City},                              
  \cny{Vietnam}                                    
}
\address[id=aff2]{%
  \orgname{Saigon Institute for Computational Science and Technology},
  \street{Quang Trung Software City, District 12},
  \city{Ho Chi Minh City},
  \cny{Vietnam}
}



\end{fmbox}


\begin{abstractbox}

\begin{abstract} 
In this paper, we study the inverse problem for a class of abstract
ultraparabolic equations which is well-known to be ill-posed. We employ
some elementary results of semi-group theory to present the formula
of solution, then show the instability cause. Since the solution exhibits
unstable dependence on the given data functions, we propose a new
regularization method to stabilize the solution. then obtain the error
estimate. A numerical example shows that the method is efficient and
feasible. This work slightly extends to the earlier results in Zouyed
et al. \cite{key-9} (2014).


\end{abstract}


\begin{keyword}
\kwd{ultraparabolic equation}
\kwd{ill-posed problem}
\kwd{semi-group method}
\kwd{stability}
\kwd{error estimate}
\end{keyword}

\begin{keyword}[class=MSC]
\kwd[47A52]{}
\kwd{20M17}
\kwd[; 26D15]{}
\end{keyword}

\end{abstractbox}
%

\end{frontmatter}



\section{Introduction}

Let us denote $\left\Vert \cdot\right\Vert $ the norm and $\left\langle \cdot,\cdot\right\rangle $
the inner product in $L^{2}\left(0,\pi\right)$, i.e.,

\[
\left\langle u,v\right\rangle =\int_{0}^{\pi}uvdx,\quad\left\Vert u\right\Vert =\sqrt{\int_{0}^{\pi}\left|u\right|^{2}dx}.
\]

In this paper, we consider the following problem: determine a function
$u:\left[0,T\right]\times\left[0,T\right]\to L^{2}\left(0,\pi\right)$
solution to the Cauchy problem

\begin{equation}
\begin{cases}
u_{t}+u_{s}-\Delta u=f\left(x,t,s\right), & \left(x,t,s\right)\in\left[0,\pi\right]\times\left[0,T\right]\times\left[0,T\right],\\
u\left(0,t,s\right)=u\left(\pi,t,s\right)=0, & \left(t,s\right)\in\left[0,T\right]\times\left[0,T\right],\\
u\left(x,T,s\right)=\psi\left(x,s\right), & \left(x,s\right)\in\left[0,\pi\right]\times\left[0,T\right],\\
u\left(x,t,T\right)=\varphi\left(x,t\right), & \left(x,t\right)\in\left[0,\pi\right]\times\left[0,T\right],
\end{cases}\label{eq:1}
\end{equation}

with corresponding perturbed data functions $\left(\psi^{\varepsilon},\varphi^{\varepsilon}\right)$
satisfying

\[
\left\Vert \psi^{\varepsilon}-\psi\right\Vert \le\varepsilon,\quad\left\Vert \varphi^{\varepsilon}-\varphi\right\Vert \le\varepsilon,
\]

where $\psi^{\varepsilon}$ and $\varphi^{\varepsilon}$ play roles
as perturbed functions and $\varepsilon>0$ represents a bound between
the exact function $\left(\varphi,\psi\right)$ and the perturbed $\left(\varphi^{\varepsilon},\psi^{\varepsilon}\right)$
over $L^{2}\left(0,\pi\right)$ and the given function $f$ is called
the source function.

Ultraparabolic equations arise in several areas of science, such as
mathematical biology in population dynamics \cite{key-13} and probability
in connection with multi-parameter Brownian motion \cite{key-17},
and in the theory of boundary layers \cite{key-12}. Due to their
applications, ultraparabolic equations have gained considerable attention
in many mathematical aspects (see, e.g., \cite{key-2,key-4,key-5,key-9,key-11,key-13}
and the references therein).

In the mathematical literature, various types of ultraparabolic problems
have been solved. There have been some papers dealing with the existence
and uniqueness of solutions for ultraparabolic equations, e.g. \cite{key-13,key-19,key-22}.
As the pioneer in numerical methods for such equations, Akrivis et
al. \cite{key-4} numerically approximated the solution of a prototype
ultraparabolic equation by applying a fixed-step backward Euler scheme
and second-order box-type finite difference method. Some extension
works for the numerical angle should be mentioned are \cite{key-21,key-23}
by A. Ashyralyev-S. Yilmaz and Michael D. Marcozzi, respectively.
We also remark that, in general, ultraparabolic equations do not possess
properties that are closely fundamental to many kinds of parabolic
equations including strong maximum principles, a priori estimates,
and so on.

In the phase of ultraparabolic ill-posed problems, the authors F.
Zouyed and F. Rebbani, very recently, proposed in \cite{key-9} the
modified quasi-boundary value method to regularize the solution of
the problem (\ref{eq:1}) in homogeneous backward case $f\equiv0$.
In particular, via the instability terms in the form of the solution
of (\ref{eq:1}) (cf. \cite[Theorem 1.1]{key-2}) they established
an approximate problem by replacing $\mathcal{A}_{\alpha}=\mathcal{A}\left(I+\alpha\mathcal{A}^{-1}\right)$
for the operator $\mathcal{A}$ and taking the perturbation $\alpha$
into final conditions of the ill-posed problem, and obtained the convergence
order $\alpha^{\theta},\theta\in\left(0,1\right)$. Motivated by that
work, this paper is devoted to investigate a new regularization method.

In the past, many approaches have been studied for solving ill-posed
problems, especially the backward heat problems. For example, Lattès
and Lions \cite{key-18}, Showalter \cite{key-24} and Boussetila
and Rebbani \cite{key-26} used quasi-reversibility method; in \cite{key-22}
Ames et al. applied the least squares method with Tikhonov-type regularization;
Clark and Oppenheimer \cite{key-15}, Denche and Bessila \cite{key-14}
and Trong et al. \cite{key-29} used quasi-boundary value method.
Moreover, some other methods should be listed are the mollification
method by Hao \cite{key-32} and the operator-splitting method studied
by Kirkup and Wadsworth \cite{key-27}. To the best of the author’s
knowledge, although there are many works on several types of parabolic
backward problems, the theoretical literature on regularizing the
inverse problems for ultraparabolic equations is very scarce. Therefore,
proposing a regularization method for the problem (\ref{eq:1}) is
the scope of this paper.

Our work presented in this paper has the following features. Firstly,
for ease of the reading, we summarize in Section 2 some well-known
facts in semi-group of operator and present the formula of the solution
of (\ref{eq:1}). Secondly, in Section 3 we construct the regularized
solution based on our method, then obtain the error estimate. Finally,
a numerical example is given in Section 4 to illustrate the efficiency
of the result.

\section{Preliminaries}

The operator $-\Delta$ is a positive self-adjoint unbounded linear
operator on $L^{2}\left(0,\pi\right)$. Therefore, it can be applied
to some elementary results in \cite{key-2,key-6,key-7,key-9}. Particularly,
the formula of the solution of the problem (\ref{eq:1}) can be obtained
by L. Lorenzi et al. \cite{key-2} and the authors in \cite{key-6,key-7}
gave a detailed description on fundamental properties of the generalized
operator. In this section, we thus recall those results in which we
want to apply to our main results in this paper. We list them and
skip their proofs for conciseness.

In fact, we shall study in this section the generalized formula of
the solution by the following operator equation in terms of semi-group
theory.

\begin{equation}
\begin{cases}
u_{t}+u_{s}+\mathcal{A}u=f\left(t,s\right), & \left(t,s\right)\in\left[0,T\right]\times\left[0,T\right],\\
u\left(T,s\right)=\psi\left(s\right), & s\in\left[0,T\right],\\
u\left(t,T\right)=\varphi\left(t\right), & t\in\left[0,T\right],
\end{cases}\label{eq:2}
\end{equation}

where $\mathcal{A}$ is a positive self-adjoint unbounded linear operator
on the Hilbert space $\mathcal{H}$.

We denote by $\left\{ E_{\lambda},\lambda>0\right\} $ the spectral
resolution of the identify associated to $\mathcal{A}$. Let us denote

\[
S\left(t\right)=e^{-t\mathcal{A}}=\int_{0}^{\infty}e^{-t\lambda}dE_{\lambda}\in\mathcal{L}\left(\mathcal{H}\right),\quad t\ge0,
\]

the $C_{0}$-semigroup of contractions generated by $-\mathcal{A}$
($\mathcal{L}\left(\mathcal{H}\right)$ stands for the Banach algebra
of bounded linear operators on $\mathcal{H}$). Then

\begin{equation}
\mathcal{A}u=\int_{0}^{\infty}\lambda dE_{\lambda}u,\label{eq:3}
\end{equation}

for all $u\in\mathcal{D}\left(\mathcal{A}\right)$. In this connection,
$u\in\mathcal{D}\left(\mathcal{A}\right)$ iff the integral (\ref{eq:3})
exists, i.e.,

\[
\int_{0}^{\infty}\lambda^{2}d\left\Vert E_{\lambda}u\right\Vert ^{2}<\infty.
\]

For this family of operators $\left\{ S\left(t\right)\right\} _{t\ge0}$
we have:

1. $\left\Vert S\left(t\right)\right\Vert \le1$ for all $t\ge0$;

2. the function $t\mapsto S\left(t\right),t>0$ is analytic;

3. for every real $r\ge0$ and $t>0$, the operator $S\left(t\right)\in\mathcal{L}\left(\mathcal{H},\mathcal{D}\left(\mathcal{A}^{r}\right)\right)$;

4. for every integer $k\ge0$ and $t>0$, $\left\Vert S^{k}\left(t\right)\right\Vert =\left\Vert \mathcal{A}^{k}S\left(t\right)\right\Vert \le c\left(k\right)t^{-k}$;

5. for every $x\in\mathcal{D}\left(\mathcal{A}^{r}\right),r\ge0$,
we have $S\left(t\right)\mathcal{A}^{r}x=\mathcal{A}^{r}S\left(t\right)x$.
\begin{rem}
In the sequel, let us denote

\[
D_{1}=\left\{ \left(t,s\right)\in\left[0,T\right]\times\left[0,T\right]:0\le s\le t\le T\right\} ;
\]

\[
D_{2}=\left\{ \left(t,s\right)\in\left[0,T\right]\times\left[0,T\right]:0\le t\le s\le T\right\} ,
\]

and make some conditions on the given functions as follows:

(A1) $\varphi\in C\left(\left[0,T\right];\mathcal{D}\left(\mathcal{A}\right)\right)\cap C^{1}\left(\left[0,T\right];\mathcal{H}\right)$;

(A2) $\psi\in C\left(\left[0,T\right];\mathcal{D}\left(\mathcal{A}\right)\right)\cap C^{1}\left(\left[0,T\right];\mathcal{H}\right)$;

(A3) $\varphi\left(0\right)=\psi\left(0\right)$;

(A4) $f\in C\left(\left[0,T\right]\times\left[0,T\right];\mathcal{H}\right)\cap C^{1}\left(D_{1}\times D_{2};\mathcal{H}\right)$.

In the following theorems, we show the formula of solution of the
problem (\ref{eq:2}) by employing Theorem 1.1 in \cite{key-2} with
$a_{1}\left(t\right)=a_{2}\left(s\right)=1$ and following the steps
in \cite{key-9}.\end{rem}
\begin{thm}
Under the conditions (A1)-(A4), the problem

\begin{equation}
\begin{cases}
u_{t}+u_{s}+\mathcal{A}u=f\left(t,s\right), & \left(t,s\right)\in\left[0,T\right]\times\left[0,T\right],\\
u\left(0,s\right)=\psi\left(s\right), & s\in\left[0,T\right],\\
u\left(t,0\right)=\varphi\left(t\right), & t\in\left[0,T\right],
\end{cases}\label{eq:4}
\end{equation}

admits a unique solution $u$ presented by the following formula.
For any $\left(t,s\right)\in D_{1}$,

\[
u\left(t,s\right)=S\left(s\right)\varphi\left(t-s\right)+\int_{0}^{s}S\left(s-\eta\right)f\left(t-s+\eta,\eta\right)d\eta,
\]

and for any $\left(t,s\right)\in D_{2}$,

\[
u\left(t,s\right)=S\left(t\right)\psi\left(s-t\right)+\int_{0}^{t}S\left(t-\eta\right)f\left(\eta,s-t+\eta\right)d\eta.
\]

Moreover, the solution $u$ belongs to the space $C\left(\left[0,T\right]\times\left[0,S\right];\mathcal{D}\left(\mathcal{A}\right)\right)\cap C^{1}\left(\left[0,T\right]\times\left[0,S\right];\mathcal{H}\right)$.
\end{thm}

\begin{thm}
Under the conditions (A1)-(A4), if the problem

\begin{equation}
\begin{cases}
u_{t}+u_{s}-\mathcal{A}u=f\left(t,s\right), & \left(t,s\right)\in\left[0,T\right]\times\left[0,T\right],\\
u\left(0,s\right)=\psi\left(s\right), & s\in\left[0,T\right],\\
u\left(t,0\right)=\varphi\left(t\right), & t\in\left[0,T\right],
\end{cases}\label{eq:5}
\end{equation}

admits a solution $u$, then this solution can be presented by

\[
u\left(t,s\right)=\begin{cases}
S^{-1}\left(t\right)\psi\left(s-t\right)+\int_{s-t}^{s}S\left(\eta-s\right)f\left(t-s+\eta,\eta\right)d\eta, & \left(t,s\right)\in D_{2},\\
S^{-1}\left(s\right)\varphi\left(t-s\right)+\int_{t-s}^{t}S\left(\eta-t\right)f\left(\eta,\eta+s-t\right)d\eta, & \left(t,s\right)\in D_{1}.
\end{cases}
\]
\end{thm}
\begin{proof}
We put $\tau=T-t,\xi=T-s$ and write

\[
u\left(t,s\right)=u\left(T-\tau,T-\xi\right):=v\left(\tau,\xi\right),
\]

the function $v\left(\tau,\xi\right):\left[0,T\right]\times\left[0,T\right]\to\mathcal{H}$
satisfies the problem (\ref{eq:4}), namely,

\[
\begin{cases}
v_{\tau}+v_{\xi}+\mathcal{A}v=F\left(\tau,\xi\right)\equiv-f\left(T-\tau,T-\xi\right), & \left(\tau,\xi\right)\in\left[0,T\right]\times\left[0,T\right],\\
v\left(0,\xi\right)=\psi_{1}\left(\xi\right)\equiv u\left(T,T-\xi\right), & \xi\in\left[0,T\right],\\
v\left(\tau,0\right)=\varphi_{1}\left(\tau\right)\equiv u\left(T-\tau,T\right), & \tau\in\left[0,T\right].
\end{cases}
\]

Thanks to Theorem 2, $v\left(\tau,\xi\right)$ is given by

\[
v\left(\tau,\xi\right)=\begin{cases}
S\left(\xi\right)\varphi_{1}\left(\tau-\xi\right)+\int_{0}^{\xi}S\left(\xi-\eta\right)F\left(\tau-\xi+\eta,\eta\right)d\eta, & \left(\tau,\xi\right)\in D_{1},\\
S\left(\tau\right)\psi_{1}\left(\xi-\tau\right)+\int_{0}^{\tau}S\left(\tau-\eta\right)F\left(\eta,\xi-\tau+\eta\right)d\eta, & \left(\tau,\xi\right)\in D_{2}.
\end{cases}
\]

It follows that

\[
u\left(t,s\right)=\begin{cases}
S\left(T-s\right)u\left(T+t-s,T\right)-\int_{0}^{T-s}S\left(T-s-\eta\right)f\left(T+t-s-\eta,T-\eta\right)d\eta, & \left(t,s\right)\in D_{2},\\
S\left(T-t\right)u\left(T,T+s-t\right)-\int_{0}^{T-t}S\left(T-t-\eta\right)f\left(T-\eta,T+s-t-\eta\right)d\eta, & \left(t,s\right)\in D_{1}.
\end{cases}
\]

Thus, we obtain

\begin{equation}
u\left(t,s\right)=\begin{cases}
S\left(T-s\right)u\left(T+t-s,T\right)-\int_{s}^{T}S\left(\zeta-s\right)f\left(t-s+\zeta,\zeta\right)d\zeta, & \left(t,s\right)\in D_{2},\\
S\left(T-t\right)u\left(T,T+s-t\right)-\int_{t}^{T}S\left(\zeta-t\right)f\left(\zeta,\zeta+s-t\right)d\zeta, & \left(t,s\right)\in D_{1},
\end{cases}\label{eq:6}
\end{equation}

by the maps $\zeta=T-\eta$ in the integrals. We can see by the initial
conditions of (\ref{eq:5}) that

\[
u\left(t,0\right)=\varphi\left(t\right)=S\left(T-t\right)u\left(T,T-t\right)-\int_{t}^{T}S\left(\zeta-t\right)f\left(\zeta,\zeta-t\right)d\zeta,
\]

\[
u\left(0,s\right)=\psi\left(s\right)=S\left(T-s\right)u\left(T-s,T\right)-\int_{s}^{T}S\left(\zeta-s\right)f\left(\zeta-s,\zeta\right)d\zeta,
\]

which leads to

\[
\begin{cases}
\varphi\left(t-s\right)=S\left(T-t+s\right)u\left(T,T-t+s\right)-\int_{t-s}^{T}S\left(\zeta-t+s\right)f\left(\zeta,\zeta-t+s\right)d\zeta, & \left(t,s\right)\in D_{1},\\
\psi\left(s-t\right)=S\left(T-s+t\right)u\left(T-s+t,T\right)-\int_{s-t}^{T}S\left(\zeta-s+t\right)f\left(\zeta-s+t,\zeta\right)d\zeta, & \left(t,s\right)\in D_{2}.
\end{cases}
\]

By virtual of semi-group properties, we get

\begin{equation}
\begin{cases}
S^{-1}\left(s\right)\varphi\left(t-s\right)=S\left(T-t\right)u\left(T,T-t+s\right)-\int_{t-s}^{T}S\left(\zeta-t\right)f\left(\zeta,\zeta-t+s\right)d\zeta, & \left(t,s\right)\in D_{1},\\
S^{-1}\left(t\right)\psi\left(s-t\right)=S\left(T-s\right)u\left(T-s+t,T\right)-\int_{s-t}^{T}S\left(\zeta-s\right)f\left(\zeta-s+t,\zeta\right)d\zeta, & \left(t,s\right)\in D_{2}.
\end{cases}\label{eq:7}
\end{equation}

Substituting (\ref{eq:7}) into (\ref{eq:6}), we thus have

\[
u\left(t,s\right)=\begin{cases}
S^{-1}\left(t\right)\psi\left(s-t\right)+\int_{s-t}^{s}S\left(\zeta-s\right)f\left(t-s+\zeta,\zeta\right)d\zeta, & \left(t,s\right)\in D_{2},\\
S^{-1}\left(s\right)\varphi\left(t-s\right)+\int_{t-s}^{t}S\left(\zeta-t\right)f\left(\zeta,\zeta+s-t\right)d\zeta, & \left(t,s\right)\in D_{1}.
\end{cases}
\]
\end{proof}
\begin{thm}
Under the conditions (A1), (A2) and (A4), if the problem (\ref{eq:2})
with $\varphi\left(T\right)=\psi\left(T\right)$ admits a solution
$u$, then this solution can be given by

\[
u\left(t,s\right)=\begin{cases}
S^{-1}\left(T-t\right)\psi\left(T+s-t\right)+\int_{t}^{T}S\left(\eta-T\right)f\left(\eta-t,\eta-s\right)d\eta, & \left(t,s\right)\in D_{1},\\
S^{-1}\left(T-s\right)\varphi\left(T+t-s\right)+\int_{s}^{T}S\left(\eta-T\right)f\left(\eta-t,\eta-s\right)d\eta, & \left(t,s\right)\in D_{2}.
\end{cases}
\]
\end{thm}
\begin{proof}
Now we put $\tau=T-t$ and $\xi=T-s$, then write

\[
u\left(t,s\right)=u\left(T-\tau,T-\xi\right):=v\left(\tau,\xi\right),
\]

the function $v\left(\tau,\xi\right):\left[0,T\right]\times\left[0,T\right]\to\mathcal{H}$
satisfies the problem (\ref{eq:5}), namely,

\[
\begin{cases}
v_{\tau}+v_{\xi}-\mathcal{A}v=F\left(\tau,\xi\right)\equiv-f\left(T-\tau,T-\xi\right), & \left(\tau,\xi\right)\in\left[0,T\right]\times\left[0,T\right],\\
v\left(0,\xi\right)=\psi_{1}\left(\xi\right)\equiv u\left(T,T-\xi\right), & \xi\in\left[0,T\right],\\
v\left(\tau,0\right)=\varphi_{1}\left(\tau\right)\equiv u\left(T-\tau,T\right), & \tau\in\left[0,T\right].
\end{cases}
\]

Using Theorem 3, the solution $v\left(\tau,\xi\right)$ can be presented
by

\[
v\left(\tau,\xi\right)=\begin{cases}
S^{-1}\left(\tau\right)\psi_{1}\left(\xi-\tau\right)+\int_{\xi-\tau}^{\xi}S\left(\eta-\xi\right)F\left(\tau-\xi+\eta,\eta\right)d\eta, & \left(\tau,\xi\right)\in D_{2},\\
S^{-1}\left(\xi\right)\varphi_{1}\left(\tau-\xi\right)+\int_{\tau-\xi}^{\tau}S\left(\eta-\tau\right)F\left(\eta,\eta+\xi-\tau\right)d\eta, & \left(\tau,\xi\right)\in D_{1}.
\end{cases}
\]

It follows that

\[
u\left(T-\tau,T-\xi\right)=\begin{cases}
S^{-1}\left(\tau\right)u\left(T,T-\xi+\tau\right)-\int_{\xi-\tau}^{\xi}S\left(\eta-\xi\right)f\left(T-\tau+\xi-\eta,T-\eta\right)d\eta, & \left(\tau,\xi\right)\in D_{2},\\
S^{-1}\left(\xi\right)u\left(T-\tau+\xi,T\right)-\int_{\tau-\xi}^{\tau}S\left(\eta-\tau\right)f\left(T-\eta,T-\eta-\xi+\tau\right)d\eta, & \left(\tau,\xi\right)\in D_{1}.
\end{cases}
\]

Hence, we obtain

\begin{eqnarray*}
u\left(t,s\right) & = & \begin{cases}
S^{-1}\left(T-t\right)\psi\left(T+s-t\right)-\int_{t-s}^{T-s}S\left(\eta-T+s\right)f\left(T+t-s-\eta,T-\eta\right)d\eta, & \left(t,s\right)\in D_{1},\\
S^{-1}\left(T-s\right)\varphi\left(T+t-s\right)-\int_{s-t}^{T-t}S\left(\eta-T+t\right)f\left(T-\eta,T+s-t-\eta\right)d\eta, & \left(t,s\right)\in D_{2},
\end{cases}\\
 & = & \begin{cases}
S^{-1}\left(T-t\right)\psi\left(T+s-t\right)-\int_{t}^{T}S\left(\zeta-T\right)f\left(T+t-\zeta,T+s-\zeta\right)d\zeta, & \left(t,s\right)\in D_{1},\\
S^{-1}\left(T-s\right)\varphi\left(T+t-s\right)-\int_{s}^{T}S\left(\zeta-T\right)f\left(T+t-\zeta,T+s-\zeta\right)d\zeta, & \left(t,s\right)\in D_{2},
\end{cases}
\end{eqnarray*}

which completes the proof.
\end{proof}
Now we return to the consideration of problem (\ref{eq:1}). All of
our results in this paper apply to more general problems, for which
the boundary conditions are generalized in Robin-type, for example,

\[
\alpha_{1}u\left(0,t,s\right)+\alpha_{2}u_{x}\left(0,t,s\right)=0,
\]

\[
\alpha_{3}u\left(\pi,t,s\right)+\alpha_{4}u_{x}\left(\pi,t,s\right)=0,
\]

or we can consider, in general, the operator equations with the self-adjoint
operator $\mathcal{A}$ having a discrete spectrum on an abstract
Hilbert space $\mathcal{H}$ and satisfying the condition that $-\mathcal{A}$
generates a compact contraction semi-group on $\mathcal{H}$, like
the problem (\ref{eq:2}) considered above. However, for the sake
of simplicity, we confine our attention to the problem (\ref{eq:1})
in which the homogeneous Dirichlet boundary conditions at the endpoints
of $\left[0,\pi\right]$ are given. In this problem, we have $\mathcal{H}=L^{2}\left(0,\pi\right)$
and $\mathcal{D}\left(\mathcal{A}\right)=H_{0}^{1}\left(0,\pi\right)\cap H^{2}\left(0,\pi\right)$,
so there exists an orthonormal basis of $L^{2}\left(0,\pi\right)$,
$\left\{ \phi_{n}\right\} _{n\in\mathbb{N}}$ satisfying (see e.g.
\cite[p. 181]{key-33})

\[
\phi_{n}\in H_{0}^{1}\left(0,\pi\right)\cap C^{\infty}\left(\left[0,\pi\right]\right),\quad\Delta\phi_{n}=-\lambda_{n}\phi_{n},\quad0<\lambda_{1}\le\lambda_{2}\le...\lim_{n\to\infty}\lambda_{n}=\infty.
\]

The Laplace operator thus has a discrete spectrum $\sigma\left(\mathcal{A}\right)=\left\{ \lambda_{n}\right\} _{n\ge1}$
with $\lambda_{n}=n^{2}$ and gives the orthonormal eigenbasis $\phi_{n}=\sqrt{\dfrac{2}{\pi}}\sin\left(nx\right)$
for $n\in\mathbb{N},n\ge1$. Then, thanks to those theorems above,
the solution has the form

\begin{equation}
u\left(x,t,s\right)=\begin{cases}
\sum_{n\ge1}\left(e^{\left(T-t\right)n^{2}}\psi_{n}\left(T+s-t\right)-\int_{t}^{T}e^{\left(T-\eta\right)n^{2}}f_{n}\left(T+t-\eta,T+s-\eta\right)d\eta\right)\sin\left(nx\right), & \left(t,s\right)\in D_{1},\\
\sum_{n\ge1}\left(e^{\left(T-s\right)n^{2}}\varphi_{n}\left(T+t-s\right)-\int_{s}^{T}e^{\left(T-\eta\right)n^{2}}f_{n}\left(T+t-\eta,T+s-\eta\right)d\eta\right)\sin\left(nx\right), & \left(t,s\right)\in D_{2},
\end{cases}\label{eq:8}
\end{equation}

where

\[
\varphi_{n}\left(t\right)=\frac{2}{\pi}\int_{0}^{\pi}\varphi\left(x,t\right)\sin\left(nx\right)dx,\quad\psi_{n}\left(s\right)=\frac{2}{\pi}\int_{0}^{\pi}\psi\left(x,s\right)\sin\left(nx\right)dx,\quad f_{n}\left(t,s\right)=\frac{2}{\pi}\int_{0}^{\pi}f\left(x,t,s\right)\sin\left(nx\right)dx.
\]

We can see that the instability is caused by all of the exponential
functions. In fact, let us see the case $\left(t,s\right)\in D_{1}$
in (\ref{eq:8}). Since the discrete spectrum increases monotonically
as $n$ tends to infinity, the rapid escalation of $e^{\left(T-t\right)n^{2}}$
and $e^{\left(T-\eta\right)n^{2}}$ is mainly the instability cause.
Even though these exact given functions $\left(\psi_{n},f_{n}\right)$
may tend to zero very fast, performing classical calculation is impossible.
It is because that the given data may be diffused by a variety of
reasons such as round-off errors, measurement errors. A small perturbation
in the data can arbitrarily generate a large error in the solution.
A regularization method is thus required.

\section{Theoretical results}

In this section, assuming that the problem has an exact solution $u$
satisfying various corresponding assumptions, we construct the regularized
solution depending continuously on the data such that converges to
the exact solution $u$ in some sense. Moreover, the accuracy of regularized
solution is estimated.

The solution of (\ref{eq:1}) can be given by

\begin{equation}
u\left(x,t,s\right)=\begin{cases}
\sum_{n\ge1}e^{\left(T-t\right)n^{2}}\left(\psi_{n}\left(T+s-t\right)-\int_{t}^{T}e^{\left(t-\eta\right)n^{2}}f_{n}\left(T+t-\eta,T+s-\eta\right)d\eta\right)\sin\left(nx\right), & \left(t,s\right)\in D_{1},\\
\sum_{n\ge1}e^{\left(T-s\right)n^{2}}\left(\varphi_{n}\left(T+t-s\right)-\int_{s}^{T}e^{\left(s-\eta\right)n^{2}}f_{n}\left(T+t-\eta,T+s-\eta\right)d\eta\right)\sin\left(nx\right), & \left(t,s\right)\in D_{2}.
\end{cases}
\end{equation}

We shall replace all instability terms by the better ones, particularly,
$\left(\varepsilon+e^{-pn^{2}}\right)^{\frac{t-T}{p}}$ and $\left(\varepsilon+e^{-pn^{2}}\right)^{\frac{s-T}{p}}$
where $p\ge1$ is a real number. Then, the regularized solution corresponding
to the exact data is

\begin{equation}
u^{\varepsilon}\left(x,t,s\right)=\sum_{n\ge1}\left(\varepsilon+e^{-pn^{2}}\right)^{\frac{t-T}{p}}\left(\psi_{n}\left(T+s-t\right)-\int_{t}^{T}e^{\left(t-\eta\right)n^{2}}f_{n}\left(T+t-\eta,T+s-\eta\right)d\eta\right)\sin\left(nx\right),\label{eq:10}
\end{equation}

for any $\left(t,s\right)\in D_{1}$, and

\begin{equation}
u^{\varepsilon}\left(x,t,s\right)=\sum_{n\ge1}\left(\varepsilon+e^{-pn^{2}}\right)^{\frac{s-T}{p}}\left(\varphi_{n}\left(T+t-s\right)-\int_{s}^{T}e^{\left(s-\eta\right)n^{2}}f_{n}\left(T+t-\eta,T+s-\eta\right)d\eta\right)\sin\left(nx\right),\label{eq:11-1}
\end{equation}

for any $\left(t,s\right)\in D_{2}$.

We also denote the regularized solution corresponding to the perturbed
data by

\begin{equation}
v^{\varepsilon}\left(x,t,s\right)=\sum_{n\ge1}\left(\varepsilon+e^{-pn^{2}}\right)^{\frac{t-T}{p}}\left(\psi_{n}^{\varepsilon}\left(T+s-t\right)-\int_{t}^{T}e^{\left(t-\eta\right)n^{2}}f_{n}\left(T+t-\eta,T+s-\eta\right)d\eta\right)\sin\left(nx\right),\label{eq:11}
\end{equation}

for any $\left(t,s\right)\in D_{1}$, and

\begin{equation}
v^{\varepsilon}\left(x,t,s\right)=\sum_{n\ge1}\left(\varepsilon+e^{-pn^{2}}\right)^{\frac{s-T}{p}}\left(\varphi_{n}^{\varepsilon}\left(T+t-s\right)-\int_{s}^{T}e^{\left(s-\eta\right)n^{2}}f_{n}\left(T+t-\eta,T+s-\eta\right)d\eta\right)\sin\left(nx\right),\label{eq:12}
\end{equation}

for any $\left(t,s\right)\in D_{2}$.

Now we shall show two elementary inequalities in the following lemmas.
\begin{lem}
For $0\le t\le T\le p$, we have

\[
\left(\varepsilon+e^{-n^{2}p}\right)^{\frac{t-T}{p}}\leq\varepsilon^{\frac{t-T}{p}}.
\]
\end{lem}
\begin{proof}
It is obvious that $\left(\varepsilon+e^{-n^{2}p}\right)^{\frac{t-T}{p}}\leq\varepsilon^{\frac{t-T}{p}}$
since $\varepsilon+e^{-n^{2}p}\ge\varepsilon$.\end{proof}
\begin{lem}
For all $x>0$, $0<\alpha<1$ we have 

\[
1-\left(x+1\right)^{-\alpha}\leq x^{1-\alpha}.
\]
\end{lem}
\begin{proof}
The proof of this lemma is based on the fact that $x^{\alpha}<\left(x+1\right)^{\alpha}$.
Therefore, we have

\begin{eqnarray*}
1+x & \leq & 1+x^{1-\alpha}\left(x+1\right)^{\alpha}\\
 & \leq & \left[1+x^{1-\alpha}\left(x+1\right)^{\alpha}\right]^{\frac{1}{\alpha}},
\end{eqnarray*}

which leads to

\begin{eqnarray*}
1-\left(x+1\right)^{-\alpha} & = & \dfrac{\left(x+1\right)^{\alpha}-1}{\left(x+1\right)^{\alpha}}\\
 & \leq & \dfrac{1+x^{1-\alpha}\left(x+1\right)^{\alpha}-1}{\left(x+1\right)^{\alpha}}\\
 & \leq & x^{1-\alpha}.
\end{eqnarray*}

\end{proof}
In the sequel, we only prove the case $\left(t,s\right)\in D_{1}$
in our main result because of the similarity. The results are about
the regularized solution depending continuously on the corresponding
data and the convergence of that solution to the exact solution. Now
we shall use two elementary lemmas above to support the proof of the
main results.
\begin{lem}
Under the conditions (A1), (A2), (A4) and assume that $\varphi\left(T\right)=\psi\left(T\right)$,
then the function $u^{\varepsilon}$ given by (\ref{eq:10})-(\ref{eq:11-1})
depends continuously on $\left(\varphi,\psi\right)$ in $L^{2}\left(0,\pi\right)$.\end{lem}
\begin{proof}
Let $u_{1}^{\varepsilon}$ and $u_{2}^{\varepsilon}$ be two solutions
of (\ref{eq:10})-(\ref{eq:11-1}) corresponding to the data $\left(\varphi^{1},\psi^{1}\right)$
and $\left(\varphi^{2},\psi^{2}\right)$, respectively. By using Parseval
relation, for $\left(t,s\right)\in D_{1}$ we have

\begin{eqnarray*}
\left\Vert u_{1}^{\varepsilon}\left(\cdot,t,s\right)-u_{2}^{\varepsilon}\left(\cdot,t,s\right)\right\Vert ^{2} & = & \frac{\pi}{2}\sum_{n\ge1}\left(\varepsilon+e^{-pn^{2}}\right)^{\frac{2\left(t-T\right)}{p}}\left(\psi_{n}^{1}\left(T+s-t\right)-\psi_{n}^{2}\left(T+s-t\right)\right)^{2}\\
 & \le & \varepsilon^{\frac{2\left(t-T\right)}{p}}\left\Vert \psi^{1}\left(T+s-t\right)-\psi^{2}\left(T+s-t\right)\right\Vert ^{2}.
\end{eqnarray*}

Similarly, for any $\left(t,s\right)\in D_{2}$, we get

\begin{eqnarray*}
\left\Vert u_{1}^{\varepsilon}\left(\cdot,t,s\right)-u_{2}^{\varepsilon}\left(\cdot,t,s\right)\right\Vert ^{2} & = & \frac{\pi}{2}\sum_{n\ge1}\left(\varepsilon+e^{-pn^{2}}\right)^{\frac{2\left(s-T\right)}{p}}\left(\varphi_{n}^{1}\left(T+t-s\right)-\varphi_{n}^{2}\left(T+t-s\right)\right)^{2}\\
 & \le & \varepsilon^{\frac{2\left(s-T\right)}{p}}\left\Vert \varphi^{1}\left(T+t-s\right)-\varphi^{2}\left(T+t-s\right)\right\Vert ^{2}.
\end{eqnarray*}
\end{proof}
\begin{thm}
Under the conditions (A1), (A2) and (A4), if the problem (\ref{eq:1})
with $\varphi\left(T\right)=\psi\left(T\right)$ admits a unique solution
$u$ satisfying

\begin{equation}
{\displaystyle \dfrac{\pi}{2}\sup_{\left(t,s\right)\in D_{1}}\sum_{n=1}^{\infty}e^{2\left(p+t-T\right)n^{2}}\left|u_{n}\left(t,s\right)\right|^{2}}<C_{1},\label{eq:14-1}
\end{equation}

and

\begin{equation}
{\displaystyle \dfrac{\pi}{2}\sup_{\left(t,s\right)\in D_{2}}\sum_{n=1}^{\infty}e^{2\left(p+s-T\right)n^{2}}\left|u_{n}\left(t,s\right)\right|^{2}}<C_{2},\label{eq:15}
\end{equation}

where ${\displaystyle u_{n}\left(t,s\right)=\int_{0}^{\pi}u\left(x,t,s\right)\sin\left(nx\right)dx}$,
let $\left(\varphi^{\varepsilon},\psi^{\varepsilon}\right)$ be perturbed
functions satisfying the conditions (A1)-(A2), respectively, and let
$v^{\varepsilon}$ be the regularized solution, given by (\ref{eq:11})-(\ref{eq:12}),
corresponding to the perturbed data $\left(\varphi^{\varepsilon},\psi^{\varepsilon}\right)$,
then  for $\left(t,s\right)\in D_{1}$ we have

\[
\left\Vert v^{\epsilon}\left(\cdot,t,s\right)-u\left(\cdot,t,s\right)\right\Vert \le\left(1+\sqrt{C_{1}}\right)\epsilon^{\frac{t-T+p}{p}},
\]

and for $\left(t,s\right)\in D_{2}$,

\[
\left\Vert v^{\epsilon}\left(\cdot,t,s\right)-u\left(\cdot,t,s\right)\right\Vert \le\left(1+\sqrt{C_{2}}\right)\epsilon^{\frac{s-T+p}{p}}.
\]
\end{thm}
\begin{proof}
For any $\left(t,s\right)\in D_{1}$, we have

\[
u\left(x,t\right)=\sum_{n\ge1}e^{\left(T-t\right)n^{2}}\left(\psi_{n}\left(T+s-t\right)-\int_{t}^{T}e^{\left(t-\eta\right)n^{2}}f_{n}\left(T+t-\eta,T+s-\eta\right)d\eta\right)\sin\left(nx\right),
\]

\[
u^{\varepsilon}\left(x,t,s\right)=\sum_{n\ge1}\left(\varepsilon+e^{-pn^{2}}\right)^{\frac{t-T}{p}}\left(\psi_{n}\left(T+s-t\right)-\int_{t}^{T}e^{\left(t-\eta\right)n^{2}}f_{n}\left(T+t-\eta,T+s-\eta\right)d\eta\right)\sin\left(nx\right),
\]

\[
v^{\varepsilon}\left(x,t,s\right)=\sum_{n\ge1}\left(\varepsilon+e^{-pn^{2}}\right)^{\frac{t-T}{p}}\left(\psi_{n}^{\varepsilon}\left(T+s-t\right)-\int_{t}^{T}e^{\left(t-\eta\right)n^{2}}f_{n}\left(T+t-\eta,T+s-\eta\right)d\eta\right)\sin\left(nx\right).
\]

Using triangle inequality, in order to get the error estimate, we
have to estimate $\left\Vert v^{\varepsilon}\left(\cdot,t,s\right)-u^{\varepsilon}\left(\cdot,t,s\right)\right\Vert $
and $\left\Vert u^{\varepsilon}\left(\cdot,t,s\right)-u\left(\cdot,t,s\right)\right\Vert $.
Indeed, we get

\begin{eqnarray}
\left\Vert v^{\varepsilon}\left(\cdot,t,s\right)-u^{\varepsilon}\left(\cdot,t,s\right)\right\Vert ^{2} & = & \frac{\pi}{2}\sum_{n\ge1}\left(\varepsilon+e^{-pn^{2}}\right)^{\frac{2\left(t-T\right)}{p}}\left(\psi_{n}^{\varepsilon}\left(T+s-t\right)-\psi_{n}\left(T+s-t\right)\right)^{2}\nonumber \\
 & \le & \varepsilon^{\frac{2\left(t-T\right)}{p}}\left\Vert \psi^{\varepsilon}\left(T+s-t\right)-\psi\left(T+s-t\right)\right\Vert ^{2}\nonumber \\
 & \le & \varepsilon^{\frac{2\left(t-T+p\right)}{p}}.\label{eq:16}
\end{eqnarray}

Next, $\left\Vert u^{\varepsilon}\left(\cdot,t,s\right)-u\left(\cdot,t,s\right)\right\Vert $
can be estimated as follows. We put

\[
u_{n}\left(t,s\right)={\displaystyle e^{\left(T-t\right)n^{2}}}\psi_{n}\left(T+s-t\right)-\int_{t}^{T}e^{\left(T-\eta\right)n^{2}}f_{n}\left(T+t-\eta,T+s-\eta\right)d\eta,
\]

then we have

\begin{eqnarray*}
\left(1+\varepsilon e^{n^{2}p}\right)^{\frac{t-T}{p}}u_{n}\left(t,s\right) & = & \left(\left(1+\varepsilon e^{n^{2}p}\right){\displaystyle e^{-pn^{2}}}\right)^{\frac{t-T}{p}}\psi_{n}\left(T+s-t\right)\\
 &  & -\int_{t}^{T}\left(1+\varepsilon e^{n^{2}p}\right)^{\frac{t-T}{p}}e^{\left(T-t\right)n^{2}}e^{\left(t-\eta\right)n^{2}}f_{n}\left(T+t-\eta,T+s-\eta\right)d\eta\\
 & = & \left(\varepsilon+e^{-pn^{2}}\right)^{\frac{t-T}{p}}\psi_{n}\left(T+s-t\right)\\
 &  & -\int_{t}^{T}\left(\varepsilon+e^{-pn^{2}}\right)^{\frac{t-T}{p}}e^{\left(t-\eta\right)n^{2}}f_{n}\left(T+t-\eta,T+s-\eta\right)d\eta.
\end{eqnarray*}

Therefore, we conclude that

\[
\left(\varepsilon+e^{-pn^{2}}\right)^{\frac{t-T}{p}}\left(\psi_{n}\left(T+s-t\right)-\int_{t}^{T}e^{\left(t-\eta\right)n^{2}}f_{n}\left(T+t-\eta,T+s-\eta\right)d\eta\right)\equiv u_{n}^{\varepsilon}\left(t,s\right)=\left(1+\varepsilon e^{n^{2}p}\right)^{\frac{t-T}{p}}u_{n}\left(t,s\right).
\]

Now using Parseval relation again, we thus obtain

\[
\left\Vert u^{\varepsilon}\left(\cdot,t,s\right)-u\left(\cdot,t,s\right)\right\Vert ^{2}=\frac{\pi}{2}\sum_{n\ge1}\left|u_{n}^{\varepsilon}\left(t,s\right)-u_{n}\left(t,s\right)\right|^{2}=\frac{\pi}{2}\sum_{n\ge1}\left(1-\left(1+\varepsilon e^{n^{2}p}\right)^{\frac{t-T}{p}}\right)^{2}\left|u_{n}\left(t,s\right)\right|^{2}.
\]

Thanks to Lemma 6 and the assumption (\ref{eq:14-1}), we have

\begin{equation}
\left\Vert u^{\varepsilon}\left(\cdot,t,s\right)-u\left(\cdot,t,s\right)\right\Vert ^{2}\le\frac{\pi}{2}\sum_{n\ge1}\left(\varepsilon e^{n^{2}p}\right)^{2-\frac{2\left(T-t\right)}{p}}\left|u_{n}\left(t,s\right)\right|^{2}\le\varepsilon^{2\left(1+\frac{t-T}{p}\right)}C_{1}.\label{eq:17}
\end{equation}

Combining (\ref{eq:16})-(\ref{eq:17}), we obtain

\begin{eqnarray*}
\left\Vert v^{\epsilon}\left(\cdot,t,s\right)-u\left(\cdot,t,s\right)\right\Vert  & \le & \left\Vert v^{\epsilon}\left(\cdot,t,s\right)-u^{\varepsilon}\left(\cdot,t,s\right)\right\Vert +\left\Vert u^{\epsilon}\left(\cdot,t,s\right)-u\left(\cdot,t,s\right)\right\Vert \\
 & \le & \varepsilon^{\frac{t-T+p}{p}}+\varepsilon^{1+\frac{t-T}{p}}\sqrt{C_{1}}\le\left(1+\sqrt{C_{1}}\right)\varepsilon^{\frac{t-T+p}{p}}.
\end{eqnarray*}

Similarly, we obtain the error estimate

\[
\left\Vert v^{\epsilon}\left(\cdot,t,s\right)-u\left(\cdot,t,s\right)\right\Vert \le\left(1+\sqrt{C_{2}}\right)\varepsilon^{\frac{s-T+p}{p}},
\]

for the case $\left(t,s\right)\in D_{2}$ with the assumption (\ref{eq:15}).

Hence, we complete the proof.\end{proof}
\begin{rem}
From Theorem 8, we can see that $v^{\varepsilon}\left(\cdot,t,s\right)$
strongly converges to $u\left(\cdot,t,s\right)$ in $L^{2}\left(0,\pi\right)$
for any $\left(t,s\right)\in\left[0,T\right]\times\left[0,T\right]$
as $\varepsilon$ tends to zero. One advantage of this method is that
the endpoints of time $\left[0,T\right]\times\left[0,T\right]$, for
example, $\left(t,s\right)=\left(0,0\right)$ and $\left(t,s\right)=\left(T,T\right)$
nearly have the same rate of convergence in some cases. Indeed, the
convergence speed at $\left(t,s\right)=\left(0,0\right)$ is $\varepsilon^{\frac{p-T}{p}}$
and it is of order $\varepsilon$ for $\left(t,s\right)=\left(T,T\right)$.
Then, if $p$ is very large for any fixed $T>0$, the order $\varepsilon^{\frac{p-T}{p}}$
may approach $\varepsilon$. This creates the globally stability behavior
of the error in numerical sense. On the other hand, the natural acceptance
of (\ref{eq:14-1})-(\ref{eq:15}) can be obtained at $\left(t,s\right)=\left(0,0\right)$.
Namely, by letting $p=T$ the conditions become

\[
{\displaystyle \dfrac{\pi}{2}\sum_{n=1}^{\infty}\left|u_{n}\left(0,0\right)\right|^{2}}=\left\Vert u\left(\cdot,0,0\right)\right\Vert ^{2}.
\]

Moreover, the error is of order $\mathcal{O}\left(\varepsilon^{\frac{p-T}{p}}\right)$
for all $\left(t,s\right)\in\left[0,T\right]\times\left[0,T\right]$.
If $p>T$, this error is faster than the order $\ln\left(\varepsilon^{-1}\right)^{-q},q>0$
as $\varepsilon\to0$ which is studied in many works, such as \cite{key-6,key-14,key-15,key-29}.
Combining the strong points above, the reader can infer that our method
is feasible.
\end{rem}

\section{A numerical example}
In order to see how well the method works, we consider as an example
the problem (\ref{eq:1}) by choosing

\[
f\left(x,t,s\right)=-2e^{-2t-s}\sin x,\quad\psi\left(x,s\right)=e^{-2-s}\sin x,\quad\varphi\left(x,t\right)=e^{-2t-1}\sin x,
\]
 and the domain $\left[0,\pi\right]\times\left[0,1\right]^{2}$. For
these given functions, the problem has a unique solution

\begin{equation}
u_{ex}\left(x,t,s\right)=e^{-2t-s}\sin x.\label{eq:13-1}
\end{equation}

Now let us take perturbation on data functions as follows. For $m\in\mathbb{N}$,
we define

\[
\varphi_{m}\left(x,t\right)=e^{-2t-1}\sin x+\frac{\sin\left(mx\right)}{m},
\]

\[
\psi_{m}\left(x,s\right)=e^{-s-2}\sin x+\frac{\sin\left(mx\right)}{m}.
\]

Thus, the solution corresponding to the perturbed data functions is

\begin{eqnarray*}
u_{m}\left(x,t,s\right) & = & \begin{cases}
e^{-s-2}\sin x+\frac{e^{\left(1-t\right)m^{2}}}{m}\sin\left(mx\right)\\
+2\int_{t}^{1}e^{1-\eta}e^{-2\left(1+t-\eta\right)-\left(1+s-\eta\right)}\sin xd\eta, & \left(t,s\right)\in D_{1},\\
e^{-2-2t+s}\sin x+\frac{e^{\left(1-s\right)m^{2}}}{m}\sin\left(mx\right)\\
+2\int_{s}^{1}e^{1-\eta}e^{-2\left(1+t-\eta\right)-\left(1+s-\eta\right)}\sin xd\eta, & \left(t,s\right)\in D_{2}.
\end{cases}\\
 & = & \begin{cases}
e^{-s-2}\sin x+\frac{e^{\left(1-t\right)m^{2}}}{m}\sin\left(mx\right)+e^{-2t-s-2}\left(e^{2}-e^{2t}\right)\sin x, & \left(t,s\right)\in D_{1},\\
e^{-2-2t+s}\sin x+\frac{e^{\left(1-s\right)m^{2}}}{m}\sin\left(mx\right)+e^{-2t-s-2}\left(e^{2}-e^{2s}\right)\sin x, & \left(t,s\right)\in D_{2}.
\end{cases}
\end{eqnarray*}

It is easy to see that $\left(\varphi_{m},\psi_{m}\right)$ converges
to $\left(\varphi,\psi\right)$ over the norm $L^{2}\left(0,\pi\right)$
as $m\to\infty$. To observe the ill-posedness, we can compute, for
example, $u_{ex}\left(x,\dfrac{1}{2},\dfrac{1}{2}\right)=e^{-\frac{3}{2}}\sin x$
and

\[
u_{m}\left(x,\frac{1}{2},\frac{1}{2}\right)=e^{\frac{-3}{2}}\sin x+\frac{e^{\frac{m^{2}}{2}}}{m}\sin\left(mx\right).
\]

Therefore, we get

\[
\left\Vert u_{m}\left(\cdot,\frac{1}{2},\frac{1}{2}\right)-u\left(\cdot,\frac{1}{2},\frac{1}{2}\right)\right\Vert ^{2}=\int_{0}^{\pi}\frac{e^{m^{2}}}{m^{2}}\sin^{2}\left(mx\right)dx=\frac{\pi e^{m^{2}}}{2m^{2}}\to\infty,
\]

as $m\to\infty$. This divergence is also showed in Figure 1 with
$m=2$ and $m=3$.

\begin{figure}
\begin{centering}
\includegraphics[scale=0.6]{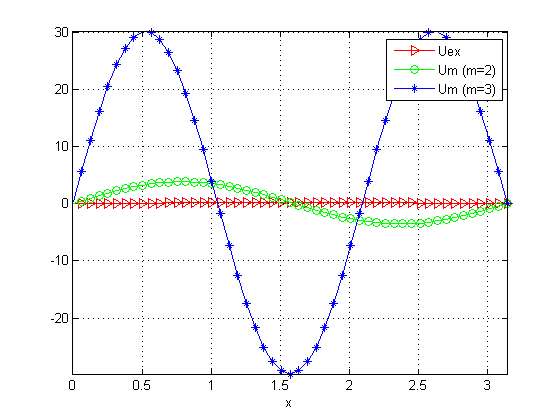}
\par\end{centering}

\centering{}\protect\caption{The exact solution $u_{ex}$ and the approximation solution within
perturbations $u_{m}$.}
\end{figure}

\begin{table}[h!]
\protect\caption{Comparison of absolute errors between the regularized solutions $v_{m}$
of $m=10^{2}$ and $m=10^{10}$.}
\begin{tabular}{cccccc}
\hline 
\multirow{2}{*}{$\left(x,t,s\right)$} & \multirow{2}{*}{Exact value} & App. value 1 & App. value 2 & \multirow{2}{*}{Abs. error 1} & \multirow{2}{*}{Abs. error 2}\tabularnewline
 &  &  ($m=10^{2}$) & ($m=10^{10}$) &  & \tabularnewline
\hline 
$\left(\dfrac{\pi}{2},0.75,0.75\right)$ & 0.1053992246 & 0.0915741799 & 0.1053992172 & 0.0138250446 & 7.4E-09\tabularnewline
\hline 
$\left(\dfrac{\pi}{2},0.5,0.5\right)$ & 0.2231301601 & 0.1684339068 & 0.2231301293 & 0.0546962533 & 3.08E-08\tabularnewline
\hline 
$\left(\dfrac{\pi}{2},0.25,0.25\right)$ & 0.4723665527 & 0.3098032761 & 0.4723664549 & 0.1625632766 & 9.87E-08\tabularnewline
\hline 
$\left(\dfrac{\pi}{2},0.125,0.125\right)$ & 0.6872892788 & 0.4201595585 & 0.6872891127 & 0.2671297203 & 1.661E-07\tabularnewline
\hline 
$\left(\dfrac{\pi}{2},0,0\right)$ & 1 & 0.5698263001 & 0.9999997239 & 0.4301736999 & 2.761E-07\tabularnewline
\hline 
\end{tabular}
\end{table}

\begin{figure}
\begin{centering}
\subfloat[Exact]{\begin{centering}
\includegraphics[scale=0.6]{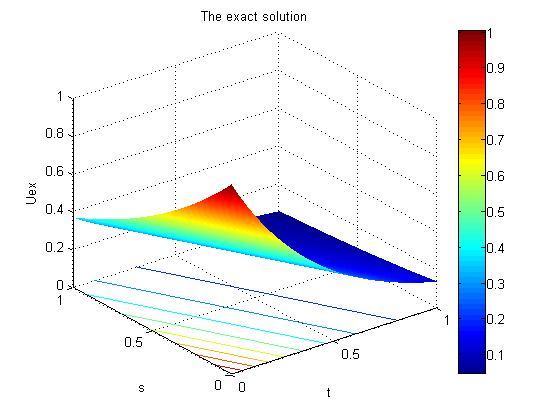}
\par\end{centering}

} 
\par\end{centering}

\begin{centering}
\subfloat[Regularized ($m=10^{10}$)]{\begin{centering}
\includegraphics[scale=0.6]{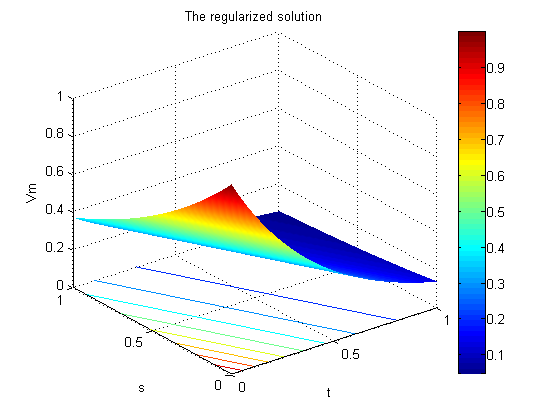}
\par\end{centering}

}
\par\end{centering}

\protect\caption{Plot of the exact and regularized solutions at the midpoint of $\left[0,\pi\right]$.}
\end{figure}

\begin{figure}
\begin{centering}
\includegraphics[scale=0.6]{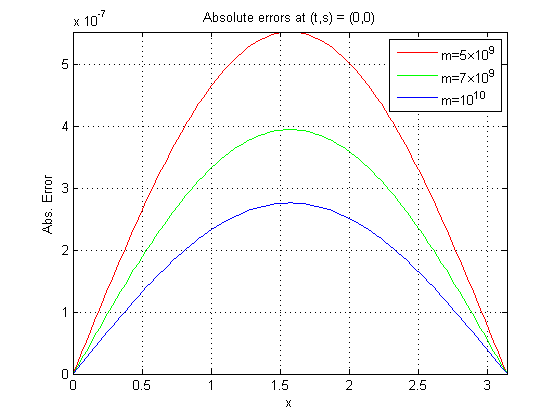}
\par\end{centering}

\protect\caption{Plot of absolute errors at the endpoint of time $\left(t,s\right)=\left(0,0\right)$.}
\end{figure}

Now we compute the regularized solution based on (\ref{eq:11})-(\ref{eq:12})
as follows.

\begin{eqnarray}
v_{m}\left(x,t,s\right) & = & \left(\sqrt{\frac{\pi}{2}}\frac{1}{m}+e^{-p}\right)^{\frac{t-1}{p}}e^{-3-s+t}\sin x+\left(\sqrt{\frac{\pi}{2}}\frac{1}{m}+e^{-pm^{2}}\right)^{\frac{t-1}{p}}\frac{\sin\left(mx\right)}{m}\nonumber \\
 &  & +2\left(\sqrt{\frac{\pi}{2}}\frac{1}{m}+e^{-p}\right)^{\frac{t-1}{p}}\int_{t}^{1}e^{t-\eta}e^{-2\left(1+t-\eta\right)-\left(1+s-\eta\right)}\sin xd\eta\nonumber \\
 & = & \left(\sqrt{\frac{\pi}{2}}\frac{1}{m}+e^{-p}\right)^{\frac{t-1}{p}}\left(e^{-3-s+t}+e^{-3-t-s}\left(e^{2}-e^{2t}\right)\right)\sin x\nonumber \\
 &  & +\left(\sqrt{\frac{\pi}{2}}\frac{1}{m}+e^{-pm^{2}}\right)^{\frac{t-1}{p}}\frac{\sin\left(mx\right)}{m}\nonumber \\
 & = & \left(\sqrt{\frac{\pi}{2}}\frac{1}{m}+e^{-p}\right)^{\frac{t-1}{p}}e^{-1-t-s}\sin x+\left(\sqrt{\frac{\pi}{2}}\frac{1}{m}+e^{-pm^{2}}\right)^{\frac{t-1}{p}}\frac{\sin\left(mx\right)}{m},\label{eq:13}
\end{eqnarray}

for any $\left(t,s\right)\in D_{1}$, and

\begin{eqnarray}
v_{m}\left(x,t,s\right) & = & \left(\sqrt{\frac{\pi}{2}}\frac{1}{m}+e^{-p}\right)^{\frac{s-1}{p}}e^{-3-2t+2s}\sin x+\left(\sqrt{\frac{\pi}{2}}\frac{1}{m}+e^{-pm^{2}}\right)^{\frac{s-1}{p}}\frac{\sin\left(mx\right)}{m},\nonumber \\
 &  & +2\left(\sqrt{\frac{\pi}{2}}\frac{1}{m}+e^{-p}\right)^{\frac{s-1}{p}}\int_{s}^{1}e^{s-\eta}e^{-2\left(1+t-\eta\right)-\left(1+s-\eta\right)}\sin xd\eta\nonumber \\
 & = & \left(\sqrt{\frac{\pi}{2}}\frac{1}{m}+e^{-p}\right)^{\frac{s-1}{p}}\left(e^{-3-2t+2s}+e^{-3-2t}\left(e^{2}-e^{2s}\right)\right)\sin x\nonumber \\
 &  & +\left(\sqrt{\frac{\pi}{2}}\frac{1}{m}+e^{-pm^{2}}\right)^{\frac{s-1}{p}}\frac{\sin\left(mx\right)}{m}\nonumber \\
 & = & \left(\sqrt{\frac{\pi}{2}}\frac{1}{m}+e^{-p}\right)^{\frac{s-1}{p}}e^{-1-2t}\sin x+\left(\sqrt{\frac{\pi}{2}}\frac{1}{m}+e^{-pm^{2}}\right)^{\frac{s-1}{p}}\frac{\sin\left(mx\right)}{m},\label{eq:14}
\end{eqnarray}

for any $\left(t,s\right)\in D_{2}$.

To obtain numerical results, we use a uniform grid of mesh-points
$\left(x,t,s\right)=\left(x_{j},t_{k},s_{m}\right)$ where

\[
x_{j}=j\Delta x,\quad\Delta x=\frac{\pi}{K},\; j=\overline{0,K},
\]

\[
t_{k}=k\Delta t,\; s_{l}=l\Delta s,\quad\Delta t=\Delta s=\frac{1}{M},\; k,l=\overline{0,M}.
\]

We thus seek the discrete solutions $u_{ex}^{j,k,l}=u_{ex}\left(x_{j},t_{k},s_{l}\right)$
and $v_{m}^{j,k,l}=v_{m}\left(x_{j},t_{k},s_{l}\right)$ given by
(\ref{eq:13-1}) and (\ref{eq:13})-(\ref{eq:14}), respectively.

By fixing $K=100,M=80$ and $p=10$, the numerical results are shown
in Table 1 and illustrated in Figs. 2-3 as below. Fig. 2 is the graphical
representations for curved surfaces of the exact solution $\left(t,s\right)\mapsto u_{ex}\left(\dfrac{\pi}{2},t,s\right)\equiv e^{-2t-s}$,
and of the approximate solution $\left(t,s\right)\mapsto v_{m}\left(\dfrac{\pi}{2},t,s\right)$
determined in (\ref{eq:13})-(\ref{eq:14}) with $m=10^{10}$. In
Fig. 3, we have drawn the exact solution $x\mapsto u_{ex}\left(x,0,0\right)\equiv\sin x$
and the approximate solution $x\mapsto v_{m}\left(x,0,0\right)$ where
$m$ are $5\times10^{9}$, $7\times10^{9}$ and $10^{10}$, respectively,
in order to see the convergence at $\left(t,s\right)=\left(0,0\right)$
as $m$ becomes very large, namely, the bound $\varepsilon$ in theoretical
result tends to zero. As in Figs. 2-3, we can conclude that the regularized
solution converges to the exact one as the error becomes smaller and
smaller. Moreover, convergence is, particularly, observed from the
absolute (abs.) errors in Table 1. Hence, our numerical results are
all reasonable for the theoretical result.

\section{Conclusion}

In this work, a regularization method has been successfully applied
to the inverse ultraparabolic problem. This method is to replace the
instability terms appearing in the formula of the solution which is
employed by semi-group theory. Therefore, such a way forms the so-called
regularized solution which strongly converges to the exact solution
in $L^{2}$-norm. We also obtain the error estimate which is of order
$\varepsilon^{\frac{p-T}{p}},p>T$. By a numerical example, application
of the method is flexible and calculation of successive approximations
is direct and straightforward. This work is more general than \cite{key-9},
a recent work of Zouyed et al., in both error estimate and the considered
problem.


\begin{backmatter}

\section*{Competing interests}
  The authors declare that they have no competing interests.

\section*{Author's contributions}
    VAK, LTL organized and wrote this manuscript. VAK, LTL and TTH contributed to all the steps of the proofs in this research together. NHT participated in the discussion and corrected the main results. All authors read and approved the final manuscript.

\section*{Acknowledgements}
  The authors wish to express their sincere
thanks to the anonymous referees and the handling editor 
for many constructive comments leading to the
improved version of this paper.
\end{backmatter}

\begin{thebibliography}{10}
\bibitem{key-1}W. H. Press et al., Numerical recipes in Fortran 90,
2nd ed., Cambridge University Press, New York, 1996.

\bibitem{key-2}Luca Lorenzi, An ultraparabolic integrodifferential
equation, Le Matematiche, Vol. LIII (1998)-Fasc. II, pp. 401-435.

\bibitem{key-3}Huy Tuan Nguyen, Quoc Viet Tran, Van Thinh Nguyen,
Some remarks on a modified Helmholtz equation with inhomogeneous source,
Applied Mathematical Modelling 37 (2013) 793-814.

\bibitem{key-4}G. Akrivis, M. Crouzeix, V. Thom\'ee, Numerical methods
for ultraparabolic equations, Calcolo, vol. 31, no. 3-4, pp. 179-190,
1994.

\bibitem{key-5}T. G. Gen\v cev, Ultraparabolic equations, Dokl.
Akad. Nauk SSSR 151 (1963), 265-268; English Transl. 

\bibitem{key-6}N. H. Tuan, D. D. Trong, P. H. Quan, On a backward
Cauchy problem associated with continuous spectrum operator, Nonlinear
Analysis 73 (2010) 1966-1972.

\bibitem{key-7}A. Pazy, Semigroups of Linear Operators and Applications
to Partial Differential Equations, Appl. Math. Sci. 44, Springer-Verlag,
New York, 1983.

\bibitem{key-8}A. Bensoussan, P. L. Chow and J. L. Lions, Filtering
theory for stochastic process with two-dimensional time parameter,
Math. Comput. Simulation 22 (1980), no. 3, 213-221.

\bibitem{key-9}F. Zouyed, F. Rebbani, A modified quasi-boundary value
method for an ultraparabolic ill-posed problem, J. Inverse Ill-posed
Probl., de Gruyter 2014.

\bibitem{key-10}G. A. Anastassiou, G. R. Goldestein and J. A. Goldstein,
Uniqueness for evolution in multidimensional time, Nonlinear Anal.
64 (2006), 33-41.

\bibitem{key-11}M. D. Francesco, A. Pascucci, A continuous dependence
result for ultraparabolic equations in option pricing, J. Math. Anl.
Appl. 336 (2007) 1026-1041.

\bibitem{key-12}S. Chandrasekhar, Stochastic problems in physics
and astronomy, Rev. Mod. Phys 15 (1943) 1-89. Reprinted in selected
papers on noise and stochastic processes (Ed. N. Wax). New York: Dover,
195.

\bibitem{key-13}A. I. Kozhanov, On the solvability of boundary value
problems for quasilinear ultraparabolic equations in some mathematical
models of the dynamics of biological systems, Journal of Applied and
Industrial Mathematics, 2010, Vol. 4, No. 4, pp. 512-525.

\bibitem{key-14}M. Denche, K. Bessila, A modified quasi-boundary
value method for ill-posed problems, J. Math. Anna. Appl. 301(2005),
419-426.

\bibitem{key-15}G. W. Clark, S. F. Oppenheimer, Quasireversibility
methods for non-well-posed problems, Electron. J. Differential Equations
1994(1994), Article 08.

\bibitem{key-16}A. Lunardi, Analytic semigroups and optimal regularity
in parabolic problems, Birkhauser Verlag, Basel, 1995.

\bibitem{key-17}G. E. Uhlenbeck, L. S. Ornstein, On the theory of
the Brownian motion, Phys. Rev. 36 (1930) 823-841.

\bibitem{key-18}R. Lattès and J. L. Lions, The Method of Quasi-Reversibility.
Applications to Partial Dierential Equations, Elsevier, New York,
1969.

\bibitem{key-19} S. A. Tersenov, Well-posedness of boundary value
problems for a certain ultraparabolic equation, Siberian Mathematical
Journal, Vol. 40, No. 6, 1999.

\bibitem{key-20}R. E. Showalter, Hilbert space methods for partial
differential equatons, Electronic Journal of Differential Equations,
Monograph 01, 1994.

\bibitem{key-21}A. Ashyralyev and S. Yilmaz, An approximation of
Ultra-Parabolic equations, Abstract and Applied Analysis, vol. 2012,
Article ID 840621, 14 pages, 2012.

\bibitem{key-22}K.A. Ames, J.F. Epperson, A kernel-based method for
the approximate solutions of backward parabolic problems, SIAM J.
Numer. Anal. 34 (1997) 1357-1390.

\bibitem{key-23}V. S. Dron' and S. D. Ivasyshen, Properties of the
fundamental solutions and uniqueness theorems for the solutions of
the Cauchy problem for one class of ultraparabolic equations, Ukrainian
Mathematical Journal, Vol. 50, No. 11, 1998.

\bibitem{key-24}Michael D. Marcozzi, Extrapolation discontinuous
Galerkin method for ultraparabolic equations, Journal of Computational
and Applied Mathematics 224 (2009) 679-687.

\bibitem{key-25}R. E. Showalter, The final value problem for evolution
equations, J. Math. Anal. Appl. 47 (1974) 563-572.

\bibitem{key-26}N. Boussetila, F. Rebbani, A modified quasi-reversibility
method for a class of ill-posed Cauchy problems, Georgian Mathematical
Journal, Volume 14 (2007), Number 4, 627-642.

\bibitem{key-27}S. M. Kirkup, M. Wadsworth, Solution of inverse diffusion
problems by operator-splitting methods, Appl. Math. Model. 26 (2002)
1003-1018.

\bibitem{key-28}V. A. Khoa, L. T. Lan, N. T. Y. Ngoc, N. H. Tuan,
A numerical approach to approximation for a nonlinear ultraparabolic
equation, arXiv.org:1408.1351 (2014).

\bibitem{key-29}D. D. Trong, P. H. Quan, T. V. Khanh, N. H. Tuan,
A nonlinear case of the 1-D backward heat problem: Regularization
and error estimate, Z. Anal. Anwend. 26 (2) (2007) 231-245.

\bibitem{key-30}D. D. Trong , N. T. Long, A. P. N. Dinh, Nonhomogeneous
heat equation: Identification and regularization for the inhomogeneous
term, J. Math. Anal. 312 (2005) 93-104.

\bibitem{key-31}P. Marcati, R. Serafini, Asymptotic behaviour in
age dependent population dynamics with spatial spread. Boll. Un. Mat.
Ital. B (5) 16 (1979), no. 2, 734-753.

\bibitem{key-32}D.N. Hao, A mollification method for ill-posed problems,
Numer. Math. 68 (1994) 469-506.

\bibitem[33]{key-33} Anna A. Kwieci\'nska, Stabilization of partial
differential equations by noise, Stochastic Processes and their Applications
79 (1999) 179-184.

\end{thebibliography}
\end{document}